\documentclass{amsart}
\usepackage[mathscr]{eucal}
\usepackage{amssymb, amsmath, array, amscd}
\usepackage{enumerate}
\usepackage{graphicx}
\usepackage{url}
\usepackage[usenames]{color}
\usepackage[colorlinks=true,linkcolor=webgreen,filecolor=webbrown,citecolor=webgreen]{hyperref}
\definecolor{webgreen}{rgb}{0,.5,0}
\definecolor{webbrown}{rgb}{.6,0,0}
\newcommand{\seqnum}[1]{\href{http://oeis.org/#1}{\underline{#1}}}

\usepackage{subfigure}

\input xy
\xyoption{all}

\newtheorem{theorem}{Theorem}[section]

\newtheorem{corollary}[theorem]{Corollary}
\newtheorem{proposition}[theorem]{Proposition}

\theoremstyle{definition}
\newtheorem{definition}[theorem]{Definition}
\newtheorem{example}[theorem]{Example}
\newtheorem{remark}[theorem]{Remark}

\newtheorem{conjecture}{Conjecture}[section]

\begin{document}

%
% Last updated: October 16, 2018
%

\title{Classes of Gap Balancing Numbers}

\author{Jeremiah Bartz}
\address{University of North Dakota, Department of Mathematics, Witmer Hall Room 313, 101 Cornell Street Stop 8376, Grand Forks, ND 58202, USA}
\email{jeremiah.bartz@und.edu}

\author{Bruce Dearden}
\address{University of North Dakota, Department of Mathematics, Witmer Hall Room 313, 101 Cornell Street Stop 8376, Grand Forks, ND 58202, USA}
\email{bruce.dearden@und.edu}
     
\author{Joel Iiams}
\address{University of North Dakota, Department of Mathematics, Witmer Hall Room 313, 101 Cornell Street Stop 8376, Grand Forks, ND 58202, USA}
     \email{joel.iiams@und.edu}

\dedicatory{}
\thanks{}

% ABSTRACT
\begin{abstract}
Gap balancing numbers are a certain generalization of balancing and cobalancing numbers that arise from studying the equation ${T(L)+T(B)=T(m)}$ where $T(i)$ is the $i$th triangular number. In this paper, we survey early results, attempt to unify existing terminology, and extend prior findings. In addition, we further investigate the structure of classes of gap balancing numbers and related sequences, presenting new results and a conjecture regarding the number of classes based on the gap size.
\end{abstract}

%%% Extra code to include the Math Subject Classification
\let\thefootnote\relax

\maketitle

%\date{\today}

%2010  Math Subject Classification
% The command \subjclass is not working 
%\subjclass[2010]{52C30, 05B30, 14Q10}

\section{Introduction}

Let $k\geq 0$ be an integer. A positive integer $B$ is called an upper $k$-gap balancing number with upper $k$-gap balancer $r$ if 
\begin{equation}\label{eq:1}
1+2+3+\dots+(B-k) =(B+1)+\dots +(B+r)
\end{equation}
for some integer $r\geq 0$. The upper $k$-gap balancing numbers are the balancing numbers (\seqnum{A001109}) and cobalancing numbers (\seqnum{A053141}) when $k=1$ and $k=0$, respectively. It follows from \eqref{eq:1} that $B$ is a upper $k$-gap balancing number if and only if  
\begin{equation}\label{eq:2}
T(B-k)+T(B)=T(B+r)
\end{equation} 
where $T(i)=\frac{i(i+1)}{2}$ is the $i$th triangular number. 

Interest in balancing numbers \cite{bal} and their generalizations \cite{tseq,tbal,K,L1,L2,circ,seq,cobal,gap,kgap,S} stems from contemporary investigations into the properties of square triangular numbers and related expressions. A central theme is studying a geometrically motivated sequences through solutions to associated Pell-like equations. In this paper, we study \eqref{eq:2} and its associated Pell-like equations to present several new results regarding the structure of classes of upper $k$-gap balancing numbers and related sequences. Along the way, we survey early results, attempt to unify existing terminology, and extend prior findings. We also give a conjecture regarding the number of classes of $k$-gap balancing numbers based on the gap size $k$. 

\section{Upper gap balancing and related numbers} 

In this section, we define upper $k$-gap balancing and related numbers and give some examples.

\begin{definition}\label{def:1}
Let $k\geq 0$ be an integer. Define a positive integer $B$ to be an upper $k$-gap balancing number if 
\begin{equation*}
1+2+3+\dots+(B-k) =(B+1)+\dots +(B+r)
\end{equation*}
for some integer $r\geq0$. We refer to $r$ as the upper $k$-gap balancer corresponding to the upper $k$-gap balancing number $B$.  
\end{definition}

It follows from \eqref{eq:1} that $B$ is an upper $k$-gap balancing number if and only if  
\begin{equation*}
T(B-k)+T(B)=T(B+r)
\end{equation*} 
where $T(i)=\frac{i(i+1)}{2}$ is the $i$th triangular number. Thus upper $k$-gap balancing numbers are the largest index of the triangular numbers occurring on the left side of \eqref{eq:2} and also the largest number of the $k$ consecutive numbers deleted to form the gap in \eqref{eq:1}.  
Solving \eqref{eq:2} for $r$ and $B$, respectively, gives
\begin{equation}\label{eq:3}
r=\frac{-(2B+1)+\sqrt{8B^2+8(1-k)B+(2k-1)^2}}{2}
\end{equation}
and
\begin{equation}\label{eq:4}
B= \frac{(2r+2k-1)+ \sqrt{8r^2+8k r+1}}{2} 
\end{equation}
where we take the positive square roots so that $r\geq 0$ and $B>0$. Thus $B$ is an upper $k$-gap balancing number with upper $k$-gap balancer $r$ if and only if ${8B^2+8(1-k)B+(2k-1)^2}$ is a perfect square if and only if $8r^2+8k r+1$ is a perfect square. Due to these latter expressions we make the following definitions.

\begin{definition}\label{def:2}
Let $B$ be an upper $k$-gap balancing number with upper $k$-gap balancer $r$. Define its upper $k$-gap Lucas-balancing number to be 
\begin{equation*}
C=\sqrt{8B^2+8(1-k)B+(2k-1)^2}
\end{equation*}
and its upper $k$-gap Lucas-balancer $\hat{r}$ to be 
\begin{equation*}
\hat{r}=\sqrt{8r^2+8k r+1}.
\end{equation*}
We refer to $(B,C)$ as an upper $k$-gap balancing pair and $(r,\hat{r})$ as its upper $k$-gap balancer pair.
\end{definition}

It follows from Definition \ref{def:2} that an upper $k$-gap balancing pair $(B,C)$ is a solution to the Pell-like equation 
\begin{equation}\label{eq:5}
y^2=8x^2+8(1-k)x+(2k-1)^2.
\end{equation}
Moreover, an upper $k$-gap balancer pair $(r,\hat{r})$ is a solution to
\begin{equation}\label{eq:6}
y^2=8x^2+8kx+1.
\end{equation}

\begin{remark}\label{rem:1}
Using the substitutions $z=2x+1-k$ and $z=2r+k$, respectively, \eqref{eq:5} and \eqref{eq:6} can be expressed as
\begin{equation}\label{eq:7}
y^2-2z^2=2k^2-1
\end{equation}
and
\begin{equation}\label{eq:8}
y^2-2z^2=-(2k^2-1).
\end{equation}
\end{remark}

We are also interested in the index of the triangular number appearing on the right side of \eqref{eq:2} and make the following definition. 

\begin{definition}\label{def:3}
The {\it counterbalancer} $m$ of an upper $k$-gap balancing number $B$ with upper $k$-gap balancer $r$ is defined to be $m=B+r$. 
\end{definition}

The next theorem collects several relationships which follow directly from equations \eqref{eq:3} and \eqref{eq:4} and the definitions given above.

\begin{theorem}\label{thm:1}
Suppose $(B,C)$ is an upper $k$-gap balancing pair with associated upper $k$-gap balancer pair $(r,\hat{r})$ and counterbalancer $m$. Then
\begin{enumerate}
\item[(a)] $r=\frac{-2B+C-1}{2}$;
\item[(b)] $\hat{r}=2B-2r+1-2k$;
\item[(c)] $\hat{r}=4B-C+2-2k$;
\item[(d)] $m=\frac{C-1}{2}$.
\end{enumerate}
\end{theorem}

\begin{example}
The identity $T(11)+T(20)=T(23)$ shows that $(20,47)$ is an upper $9$-gap balancing pair with upper $9$-gap balancer pair $(3,17)$ and counterbalancer 23. The upper $0$-gap and $1$-gap balancing numbers are cobalancing \cite{cobal} and balancing numbers \cite{bal}, respectively. In addition, the upper $1$-gap balancing numbers are the sequence balancing numbers \cite{seq} corresponding to the sequence of positive integers $((i))_{i\geq 1}$ and the $(a,b)$-type balancing numbers \cite{K} with $a=b=1$. More generally, the upper $k$-gap balancing numbers are shifts of the $t$-sequence balancing numbers \cite{tseq} of $((i))_{i\geq 1}$ with $t=k$. 
\end{example}

\begin{example}\label{ex:2}
The identity $T(0)+T(k)=T(k)$ shows that $(k,2k+1)$ is an upper $k$-gap balancing pair with upper $k$-gap balancer pair $(0,1)$ and counterbalancer 0. In consideration of \eqref{eq:2}, we take $k$ to be the smallest upper $k$-gap balancing number and consider $0$ to be an upper $0$-gap balancing number.
\end{example}

\begin{example}\label{ex:3}
For $k>0$ the identity $T(3k-3)+T(4k-3)=T(5k-4)$ shows $(4k-3,10k-7)$ is an upper $k$-gap balancing number with upper $k$-gap balancer pair $(k-1,4k-3)$ and counterbalancer $5k-4$.
\end{example}

\begin{remark} 
Various terminology has been used when studying \eqref{eq:1} and closely related equations. We attempt to unify the existing terminology in our choice of nomenclature. The smallest index, namely $B-k$, of the triangular numbers appearing on the left side of \eqref{eq:2} can be considered a lower $k$-gap balancing number. Let $L=B-k$. Dash, Ota, and Dash \cite{tbal} referred to such $L$ as $t$-balancing numbers where $t=k$ is the gap size. Alternatively, Panda and Rout \cite{gap} defined $k$-gap balancing numbers using the parity of $k$. When $k$ is odd, a positive integer $g_k$ is a $k$-gap balancing number if 
\begin{equation*}
1+2+\dots+\left(g_k-\frac{k+1}{2}\right)=\left(g_k+\frac{k+1}{2}\right)+\left(g_k+\frac{k+3}{2}\right)+\dots+(g_k+r_k)
\end{equation*}
for some integer $r_k\geq1$. When $k$ is even, a positive integer $g_k=2n+1$ is a $k$-gap balancing number if
\begin{equation*}
1+2+\dots+\left(n-\frac{k}{2}\right)=\left(n+\frac{k}{2}+1\right)+\left(n+\frac{k}{2}+2\right)+\dots+(n+r_k)
\end{equation*}
for some integer $r_k\geq 0$. Then $g_k$ is the median of the deleted gap when $k$ is odd, and $g_k$ is the sum of the two surviving numbers forming the edge of the gap when $k$ is even. These three notions of gap balancing numbers are related. Specifically, a $k$-gap balancing number $g_k$ with $k$-gap balancer $r_k$ corresponds to the upper $k$-gap balancing number $B$ with upper $k$-gap balancer $r$ by
\begin{equation}\label{eq:9}
B=\begin{cases}
g_k+\frac{k+1}{2}, &\text{if $k$ is odd;} \\ \frac{g_k-1}{2}+\frac{k}{2}, &\text{if $k$ is even;} 
\end{cases} \qquad \textrm{and} \qquad r=\begin{cases} r_k-\frac{k-1}{2}, & \text{if $k$ is odd;} \\ r_k-\frac{k}{2}, & \text{if $k$ is even.}\end{cases}
\end{equation}
From \eqref{eq:2} we see that a lower $k$-gap balancing number $L$ corresponds to the upper $k$-gap balancing number $B=L+k$ and the corresponding lower and upper $k$-gap balancers coincide. It is convenient for our purposes to study $\eqref{eq:2}$ using upper $k$-gap balancing numbers since they appear as indices of triangular numbers with balancing and cobalancing numbers being special cases. Nonetheless, results about upper $k$-gap balancing numbers correspond to results about lower $k$-gap and $k$-gap balancing numbers and vice versa via \eqref{eq:9} and the relation $B=L+k$. 
\end{remark}

\section{Classes of upper gap balancing and related numbers}

We present functions which generate a class of upper $k$-gap balancing and balancer pairs from known ones.
\begin{theorem}\label{thm:2}
Let $y=\sqrt{8x^2+8(1-k)x+(2k-1)^2}$ and 
\begin{equation}\label{eq:10}
\begin{bmatrix} f_k(x) \\ g_k(x) \end{bmatrix} = \begin{bmatrix} 3 & 1 \\ 8 & 3 \end{bmatrix} \begin{bmatrix} x \\ y \end{bmatrix} + \begin{bmatrix} 1-k \\ 4-4k \end{bmatrix}.
\end{equation}
If $(x,y)$ is an upper $k$-gap balancing pair, then so is $(f_k(x),g_k(x))$. 
\end{theorem}

\begin{proof} The identity
\begin{equation*}
8f_k^2(x)+8(1-k)f_k(x)+(2k-1)^2=(8x+3\sqrt{8x^2+8(1-k)x+(2k-1)^2}+4-4k)^2
\end{equation*}
shows $f_k(x)$ is an upper $k$-gap balancing number with associated upper $k$-gap Lucas-balancing number $g_k(x)$ by the comments following Definition \ref{def:1}.
\end{proof}

\begin{theorem}\label{thm:3}
Let $y=\sqrt{8x^2+8k x+1}$ and 
\begin{equation}\label{eq:11}
\begin{bmatrix} F_k(x) \\ G_k(x) \end{bmatrix} = \begin{bmatrix} 3 & 1 \\ 8 & 3 \end{bmatrix}\begin{bmatrix} x \\ y \end{bmatrix}+\left[\begin{array}{r} k \\ 4k \end{array}\right].
\end{equation}
If $(x,y)$ is an upper $k$-gap balancer pair, then so is $(F_k(x),G_k(x))$.
\end{theorem}

\begin{proof}
The identity
\begin{equation*}
8F_k^2(x)+8k F_k(x)+1=(8x+3\sqrt{8x^2+8k x+1}+4k)^2
\end{equation*}
shows $F_k(x)$ is an upper $k$-gap balancer with associated upper $k$-gap Lucas-balancer $G_k(x)$. 
\end{proof}

Alternatively, the functions in Theorems \ref{thm:2} and \ref{thm:3} can be derived computing a class of solutions to \eqref{eq:7} and \eqref{eq:8} and rewriting in terms of $x$ and $y$. Dash et al.\ \cite{tbal} exhibit similar functions to those in Theorems \ref{thm:2} and \ref{thm:3} for lower $k$-gap balancing and balancer pairs. The next theorem shows that these functions for the upper $k$-gap balancing and balancer pairs work in tandem, a result tacitly understood but not proven previously in the literature.

\begin{theorem}\label{thm:4}
If $(r,\hat{r})$ is the upper $k$-gap balancer pair associated to the upper $k$-gap balancing pair $(B,C)$, then $(F_k(r),G_k(r))$ is the upper $k$-gap balancer pair associated to the upper $k$-gap balancing pair $(f_k(B),g_k(B))$.
\end{theorem}

\begin{proof} Since $C$ and $\hat{r}$ are expressions in terms of $k$, $B$, and $r$ respectively, it suffices to show $F_k(r)$ is the upper $k$-gap balancer associated to the upper $k$-gap balancing number $f_k(B)$.
Using Theorem \ref{thm:1} with \eqref{eq:11}, we see
\begin{equation*}
F_k(r)=3r+\hat{r}+k=\frac{2B+C+1-2k}{2}.
\end{equation*}
On the other hand, the balancer of $f_k(B)$ using Theorem \ref{thm:1} is 
\begin{equation*}
\frac{-2f_k(B)+g_k(B)-1}{2}=\frac{2B+C+1-2k}{2}
\end{equation*}
since $f_k(B)=3B+C+1-k$ and $g_k(B)=8B+3C+4-4k$ using \eqref{eq:10}. The result now follows.
\end{proof}

Since $f_k$ and $g_k$ are strictly increasing on $[k,\infty)$, their inverses $f^{-1}_k$ and $g^{-1}_k$ exist. Similarly we see that $F_k$ and $G_k$ are strictly increasing on $[0,\infty)$, so their inverses $F^{-1}_k$ and $G^{-1}_k$ exist. A straightforward computation shows
\begin{equation}\label{eq:12}
\begin{bmatrix} f^{-1}_k(x) \\ g^{-1}_k(x) \end{bmatrix} = \left[\begin{array}{rr} 3 & -1 \\ -8 & 3 \end{array}\right] \begin{bmatrix} x \\ y \end{bmatrix} + \begin{bmatrix} 1-k \\ 4k-4 \end{bmatrix}
\end{equation}
and
\begin{equation*}
\begin{bmatrix} F^{-1}_k(x) \\ G^{-1}_k(x) \end{bmatrix}= \left[\begin{array}{rr} 3 & -1 \\ -8 & 3 \end{array}\right] \begin{bmatrix} x \\ y \end{bmatrix} + \left[\begin{array}{r} k \\ -4k \end{array}\right].
\end{equation*}

Given an upper $k$-gap balancing pair $(x_0,y_0)$ the functions $f_k$, $g_k$, and their inverses can be applied iteratively to form a class of solutions $((x_i,y_i))_{i\in \mathbb{Z}}$ to \eqref{eq:5}. Moreover, $F_k$, $G_k$, and their inverses can be used iteratively to form the class of associated upper $k$-gap balancer pairs by Theorem \ref{thm:4}.

The On-line Encyclopedia of Integer Sequences \cite{OEIS} contains several sequences related to upper $k$-gap balancing numbers. These include the upper 2-gap Lucas-balancing numbers (\seqnum{A077443}), 2-gap counterbalancers (\seqnum{A124124}), upper 2-gap Lucas-balancer numbers (\seqnum{A077446}), upper 5-gap Lucas-balancing numbers (\seqnum{A275797}), and upper 5-gap Lucas-balancer numbers (\seqnum{A076293}).

\section{Structure of classes of upper gap balancing numbers}

In this section, we present three new theorems and a conjecture regarding the structure of classes of upper $k$-gap balancing numbers. We begin by showing that it is possible to reindex each class of solutions to \eqref{eq:5} so that the terms with nonnegative index correspond to upper $k$-gap balancing numbers.

\begin{proposition}\label{prop:5}
Every upper $k$-gap balancing pair lies in a class of solutions $((x_{i},y_i))_{i \in \mathbb{Z}}$ of \eqref{eq:5} generated by applying \eqref{eq:10} and \eqref{eq:12} to a unique upper $k$-gap balancing pair $(x_0,y_0)$ satisfying $k \leq x_0 < 4k+2$. Moreover $(x_i,y_i)$ is an upper $k$-gap balancing pair whenever $i\geq0$.
\end{proposition}

\begin{proof} Assume $(B,C)$ is an upper $k$-gap balancing pair. Let $((b_i,c_i))_{i\in \mathbb{Z}}$ be the class of solutions to \eqref{eq:5} generated by $(b_0,c_0)=(k,2k+1)$ via $(b_{i+1},c_{i+1})=(f_k(b_i),g_k(b_i))$. Then $B\geq k$ by minimality of $b_0$. Suppose $b_{k}\leq B< b_{k+1}$ for some $k\geq 0$. Applying $f^{-1}_k$ $k$ times we obtain $b_0\leq (f^{-1}_k)^{(k)}(B)<b_1=4k+2$. Thus $(B,C)$ lies in the class of solutions generated by the upper $k$-gap balancing pair $(x_0,y_0)$ where $x_0= (f^{-1}_k)^{(k)}(B)$ and $y_0=(g_k^{-1})^{(k)}(B)$.

To establish uniqueness, suppose $(B,C)$ also lies in the class of solutions generated by the upper $k$-gap balancing pair $(x_0',y_0')$ with $k\leq x_0'<4k+2$. Then $B=f_k^{(i)}(x_0')$ for some $i\geq0$ so $b_{k}\leq f_k^{(i)}(x_0')<b_{k+1}$. Apply $f^{-1}_k$ $i$ times to see 
$b_{k-i}\leq x_0'<b_{k+1-i}$. Thus $i=k$ which implies $x_0=x_0'$ and $y_0=y_0'$ since $f^{(k)}_k$ is strictly increasing on $[k, \infty)$. 
\end{proof}

The last proposition shows that upper $k$-gap balancing numbers lying in ${[k, 4k+2)}$ can be used to generate all upper $k$-gap balancing numbers using \eqref{eq:10}. Panda and Ray \cite{cobal} showed that all cobalancing numbers lie in the class generated by the upper $0$-gap balancing pair $(0,1)$. Similarly, Behera and Panda \cite{bal} showed that all balancing numbers lie in the class generated by the upper $1$-gap balancing pair $(1,3)$. For the remainder of the paper, a class of upper $k$-gap balancing pairs $((B_i,C_i))_{i\geq0}$ refers to the sequence of upper $k$-gap balancing pairs generated from $(B_0,C_0)$ using \eqref{eq:10} and \eqref{eq:12} with $k\leq B_0<4k+2$. 

For $k>1$, determining the number of classes of upper $k$-gap balancing numbers generated using \eqref{eq:10} is a subtle problem as recognized by Rout and Panda \cite{kgap}. Dash et al.~\cite{tbal} considered cases when there are at most two classes. Fortunately the search for initial upper $k$-gap balancing pairs $(x_0,y_0)$ guaranteed by Proposition \ref{prop:5} can be made more efficient. For $k>0$, observe $(x_{-1},y_{-1})=(f^{-1}_k(x_0),g^{-1}_k(x_0))$ are also solutions of \eqref{eq:5} with $0\leq x_{-1}<k$ and $y_{-1}>0$. Thus the number of classes is the same as the number of integer solutions $(x_{-1},y_{-1})$ of \eqref{eq:5} with  $0\leq x_{-1}<k$ and $y_{-1}>0$. For convenience we refer to such a pair $(x_{-1},y_{-1})$ as the {\it seed} of the class of upper $k$-gap balancing numbers $((x_i,y_i))_{i\geq 0}$.

Rout and Panda \cite{kgap} exhibited the existence of two classes of $k$-gap balancing numbers for $k>1$ and a third class when $2k^2-1$ is a perfect square. Dash et al.~\cite{tbal} also observed these first two classes. We present these results in terms of upper $k$-gap balancing numbers in the following three examples.

\begin{example}\label{ex:4}
For $k\geq 0$, a class of upper $k$-gap balancing pairs is generated from the seed $(0, 2k-1)$. The initial upper $k$-gap balancing pair is given in Example \ref{ex:2}. \end{example}

\begin{example}\label{ex:5}
For $k>0$ the seed $(k-1,2k-1)$ generates a class of upper $k$-gap balancing pairs. The initial upper $k$-gap balancing pair is given in Example \ref{ex:3}. This class coincides with the class presented in Example \ref{ex:4} when $k=1$ .
\end{example}

\begin{example}\label{ex:6}
When $2k^2-1$ is a perfect square, then $k$ is odd and $\left(\frac{k-1}{2},\sqrt{2k^2-1}\,\right)$ is a solution in the integers to \eqref{eq:5}. Thus we obtain the class of upper $k$-gap balancing pairs generated from the seed $(\frac{k-1}{2},\sqrt{2k^2-1})$. For $k=1$, this class coincides with the classes presented in Examples \ref{ex:4} and \ref{ex:5}. The initial values for which this class appears are $k=1,5,29,169,985$.
\end{example}

Examples $\ref{ex:4}$ and $\ref{ex:5}$ show there are at least two classes of upper $k$-gap balancing numbers when $k>1$. Depending on the value of $k>1$, there can be additional classes. By the prior discussion, there are at most $k$ candidates for seeds when $k>0$. These observations lead to the following elementary upper bound for the number of classes.

\begin{theorem}
The number of classes of upper $k$-gap balancing numbers is at most $\max\{1,k\}$.
\end{theorem}

The theory of Pell equations (cf. Nagell \cite[pp.~204--205]{N}) shows that classes of solutions appear in pairs of conjugate classes. The classes described in Examples $\ref{ex:4}$ and $\ref{ex:5}$ form such a conjugate pair and are said to be {\it conjugates} of each other.  It is possible to describe conjugate pairs for \eqref{eq:5} more precisely in terms of their seeds. We use this description to show that the classes of upper $k$-gap balancer pairs corresponding to conjugate classes of upper $k$-gap balancing pairs are conjugate.

\begin{theorem}\label{thm:7}
Let $k>0$. If $(x,y)$ is a seed to \eqref{eq:5}, then $(k-1-x,y)$ is also a seed. Furthermore, the corresponding classes of solutions are conjugate. 
\end{theorem}

\begin{proof} 
The identity
\begin{equation*}
8(k-1-x)^2+8(1-k)(k-1-x)+(2k-1)^2=8x^2+8(1-k)x+(2k-1)^2
\end{equation*}
shows that $(k-1-x,y)$ is a solution to \eqref{eq:5} whenever $(x,y)$ is a solution. Furthermore $0\leq x<k$ since $(x,y)$ is a seed which implies $0\leq k-1-x<k$. Thus $(k-1-x,y)$ is also a seed. 

Using the substitution given in Remark \ref{rem:1}, the seeds $(x,y)$ and $(k-1-x,y)$ correspond to the solutions $(y,z)$ and $(y,-z)$ of \eqref{eq:7} whose classes are conjugate by the theory of Pell equations \cite[p.~205]{N}. It follows that the classes corresponding to $(x,y)$ and $(k-1-x,y)$ are conjugate. 
\end{proof}

\begin{theorem}\label{thm:8}
The classes of upper $k$-gap balancer pairs corresponding to conjugate classes of upper $k$-gap balancing pairs are conjugate.
\end{theorem}

\begin{proof}
The result holds for $k=0$ since there is a unique class of upper 0-gap balancing pairs which is conjugate to itself. Suppose $k>0$. Let $(x,y)$ and ${(k-1-x,y)}$ be the seeds associated to the conjugate classes of upper $k$-gap balancing pairs. Applying Theorem \ref{thm:1} to these seeds yields the associated values of $(r,\hat{r})$ as
\begin{equation*}
P_1=\left(\frac{-2x+y-1}{2},4x-y+2-2k\right)
\end{equation*}
and
\begin{equation*}P_2=\left(\frac{2x+y+1-2k}{2},-4x-y-2+2k\right)
\end{equation*}
respectively. Using \eqref{eq:11} on $P_2$ gives the pair 
\begin{equation*}
P_2'=\left(\frac{-2x+y+1}{2},-(4x+y+2-2k)\right)
\end{equation*}
which lies the same class as $P_2$. Employing the substitution given in Remark \ref{rem:1} shows $P_1$ and $P_2'$ correspond to the solutions 
\begin{equation*}
\left(4x-y+2-2k,-2x+y-1+k\right) \ \textrm{and} \left(-(4x-y+2-2k),-2x+y-1+k \right)
\end{equation*}
of \eqref{eq:8}, respectively. Since $-(2k^2-1)<0$ in \eqref{eq:8} it follows that the classes of upper $k$-gap balancer pairs which contain $P_1$ and $P'_2$, respectively, are conjugate by the theory of Pell equations. 
\end{proof}

When the classes of a conjugate pair of solutions coincide, the class is referred to as an {\it ambiguous} class. The class from Example $\ref{ex:6}$ is an example of an ambiguous class. By the proof of Theorem $\ref{thm:7}$, conjugate classes of solutions to \eqref{eq:5} are ambiguous when $x=k-1-x$ or equivalently $x=\frac{k-1}{2}$. Thus the class described in Example \ref{ex:6} is the only possible ambiguous class for \eqref{eq:5} when $k>0$. It follows that there is an even number of classes of upper $k$-gap balancing numbers whenever $k>1$ and $2k^2-1$ is not a perfect square. On the other hand, there are an odd number of classes of upper $k$-gap numbers when $k=0$ or $2k^2-1$ is a perfect square. Note that Theorem $\ref{thm:7}$ also reduces the search for seeds. For $k>0$, seeds consist of integer solutions $(x_{-1},y_{-1})$ of \eqref{eq:5} with $0\leq x_{-1}\leq \frac{k-1}{2}$ and $y_{-1}>0$ along with their conjugate seeds.

There appears to be no uniform upper bound on the number of possible classes for all $k\geq 0$. Table \ref{tab:1} lists the smallest value of $k$ with $n$ classes of upper $k$-gap balancing numbers which appear for $k\leq 10000$. Numerical evidence suggests the following conjecture.

\begin{table}\caption{Smallest $k$ with $n$ classes of upper $k$-gap balancing numbers for selected $n$.}\label{tab:1}
$$\begin{array}{r|rrrrrrrrrrrrrrr}
n & 1 & 2 & 3 & 4 & 6 & 8 & 9 & 10 & 12 & 16 & 18 & 20 & 24 & 32 & 48 \\ \hline
k & 0 & 2 & 5 & 9 & 44 & 37 & 985 & 1083 & 152 & 275 & 1034 & 3719 & 779 & 3414 & 8335
\end{array}$$
\end{table}

\begin{conjecture} \label{conj:9}
{\sl The number of classes of upper $k$-gap balancing numbers is equal to the number of positive divisors of $2k^2-1$.}
\end{conjecture}

It is known that Conjecture \ref{conj:9} is true for odd $k\geq 2$ such that $2k^2-1$ is prime as a consequence of the work of Tekcan, Tayat, and \"Ozbek \cite{T}. 

\section{Transition functions}

Panda and Rout \cite{gap} exhibited functions between the two classes of $2$-gap balancing numbers. In this section, we present general functions which map upper $k$-gap balancing and balancer pairs from one class to another class. Moreover, we establish that a transition function between two classes coincides with the transition function between the two corresponding conjugate classes in reversed order. As before, we let $y=\sqrt{8x^2+8(1-k)x+(2k-1)^2}$.

\begin{theorem}\label{thm:10} 
Let $((B_i,C_i))_{i\geq 0}$ and $((B'_i,C'_i))_{i\geq 0}$ be two classes of upper $k$-gap balancing pairs generated by \eqref{eq:10}. Then the transition function given by 

\begin{equation}\label{eq:13}
\begin{bmatrix} t(x) \\ \hat{t}(x) \end{bmatrix} = \begin{bmatrix} a & b \\ 8b & a \end{bmatrix} \begin{bmatrix} x \\ y \end{bmatrix}+ \begin{bmatrix} c \\ (4-4k)b \end{bmatrix}
\end{equation}
where
\begin{align}\label{eq:14}
a  &=  -\frac{8B_0B_0'+4(1-k)(B_0+B'_0)+(2k-1)^2-C_0C'_0-2k^2+1}{2k^2-1}   \\
& \notag \\
b  &=  \frac{2(C_0B'_0-B_0C'_0)+(1-k)(C_0-C'_0)}{2(2k^2-1)} \notag \\
& \notag \\
c  &= \frac{(1-k)\left[8B_0(B_0-B'_0)+4(1-k)(B_0-B'_0)-C_0(C_0-C'_0)\right]}{2(2k^2-1)} \notag  
\end{align}
maps $(B_i,C_i)$ to $(B'_i,C_i')$ for all $i$. 
\end{theorem}

\begin{proof}
Assume $t$ has the form $t(x)=ax+by+c$ for unknown coefficients $a$, $b$, and $c$. The only possible values of the coefficients can be determined by solving the system 
\begin{equation*}
\left[\begin{array}{lll}B_0 & C_0 & 1 \\ B_1 & C_1 & 1 \\ B_2 & C_2 & 1 \end{array}\right] \begin{bmatrix} a \\ b \\ c \end{bmatrix} = \begin{bmatrix} B_0' \\ B_1' \\ B_2' \end{bmatrix}
\end{equation*}
for $a$, $b$, and $c$.
 Put $(x,y)=(B_0,C_0)$ and $(z,w)=(B_0',C_0')$. Using \eqref{eq:5} this becomes 
\begin{equation}\label{eq:15}
\left[\begin{array}{lll} x & y & 1 \\ 3x+y+1-k & 8x+3y+4-4k & 1 \\ 17x+6y+8-8k & 48x+17y+24-24k & 1 \end{array}\right] \begin{bmatrix} a \\ b \\ c \end{bmatrix} = \left[\begin{array}{l} z \\ 3z+w+1-k \\ 17z+6w+8-8k \end{array}\right].
\end{equation} 
Using Cramer's rule and the identity $y^2=8x^2+8(1-k)x+(2k-1)^2$, we obtain \eqref{eq:14}. By construction, $t(B_i)=B_i'$ for $i=0,1,2$. We now show that  $t(B_i)=B_i'$ for all $i$ using induction. Assume $t(B_i)=B_i'$ for all $0\leq i \leq k$ for some $k\geq 2$. Consider the system
\begin{equation}\label{eq:16}
\left[\begin{array}{lll}B_{k-1} & C_{k-1} & 1 \\ B_{k} & C_{k} & 1 \\ B_{k+1} & C_{k+1} & 1 \end{array}\right] \begin{bmatrix} a' \\ b' \\ c' \end{bmatrix} = \left[\begin{array}{l} B_{k-1}' \\ B_{k}' \\ B_{k+1}' \end{array}\right].
\end{equation}
By induction \eqref{eq:15} holds for $(x,y)=(B_{k-2},C_{k-2})$ and $(z,w)=(B'_{k-2},C'_{k-2})$. Again using \eqref{eq:5}, the system \eqref{eq:16} can be written as
\begin{equation*}
\left[\begin{array}{lll}  3x+y+1-k & 8x+3y+4-4k & 1 \\ 17x+6y+8-8k & 48x+17y+24-24k & 1 \\ 99x+35y+49-49k & 280x+99y+140-140k & 1 \end{array}\right] \begin{bmatrix} a' \\ b' \\ c' \end{bmatrix}=\left[\begin{array}{l}  3z+w+1-k \\ 17z+6w+8-8k \\ 99z+35w+49-49k \end{array}\right].
\end{equation*}
Using Cramer's rule and the identity $y^2=8x^2+8(1-k)x+(2k-1)^2$, we solve this last system to see $a'=a$, $b'=b$, and $c'=c$. Thus $t(B_i)=B_i'$ for all $0\leq i \leq k+1$. The proof for $\hat{t}$ is similar. 
\end{proof}

The techniques used in Theorem \ref{thm:10} can be applied to upper $k$-gap balancer pairs to obtain the following result.

\begin{theorem} 
Let $((r_i,\hat{r}_i))_{i\geq 0}$ and $((r'_i,\hat{r}'_i))_{i\geq 0}$ be upper $k$-gap balancer pairs corresponding to the two classes of upper $k$-gap balancing pairs $((B_i,C_i))_{i\geq 0}$ and $((B'_i,C'_i))_{i\geq 0}$, respectively. Let $y=\sqrt{8x^2+8kx+1}$. Then the transition function given by 

\begin{equation*}
\begin{bmatrix} T(x) \\ \widehat{T}(x) \end{bmatrix} = \begin{bmatrix} a & b \\ 8b & a \end{bmatrix} \begin{bmatrix} x \\ y \end{bmatrix}+ \begin{bmatrix} c \\ 4kb \end{bmatrix}
\end{equation*}
where
\begin{align*}
a  &=  \frac{8r_0r'_0+4k(r_0+r'_0)+2k^2-\hat{r}_0\hat{r}'_0}{2k^2-1}   \\
&  \\
b  &=  \frac{2(r_0\hat{r}'_0-\hat{r}_0r'_0)+k(\hat{r}'_0-\hat{r}_0)}{2(2k^2-1)}  \\
& \\
c  &= \frac{k\left[8r_0(r'_0-r_0)+4k(r'_0-r_0)+\hat{r}_0(\hat{r}_0-\hat{r}'_0)\right]}{2(2k^2-1)}  
\end{align*}
maps $(r_i,\hat{r}_i)$ to $(r'_i,\hat{r}_i')$ for all $i$. 
\end{theorem}

The next result gives a property of transition functions when the underlying classes are conjugate.

\begin{theorem}\label{thm:12}
Let $((B_i,C_i))_{i\geq 0}$ and $((B'_i,C'_i))_{i\geq 0}$ be two classes of upper $k$-gap balancing pairs. Suppose $((\overline{B}_i,\overline{C}_i))_{i\geq 0}$ and $((\overline{B}'_i,\overline{C}'_i))_{i\geq 0}$ are their conjugate classes, respectively. Then the transition functions $(B_i,C_i)\mapsto (B'_i,C'_i)$ and $(\overline{B}'_i,\overline{C}'_i)\mapsto (\overline{B}_i,\overline{C}_i)$ are the same. Moreover, the transition functions for the corresponding classes of upper $k$-gap balancer pairs are also the same.
\end{theorem}

\begin{proof}
It follows from the proof of Theorem \ref{thm:10} that it suffices to use the seeds of two classes of upper $k$-gap balancing numbers to compute the coefficients of the transition functions between them. Observe that taking $(B_0,C_0)$ as ${(k-1-\overline{B}'_{-1},\overline{C}'_{-1})}$ and $(B'_0,C_0')$ as $(k-1-\overline{B}_{-1},\overline{C}_{-1})$ in Theorem \ref{thm:10} gives the same values of coefficients \eqref{eq:14} of the transition function \eqref{eq:13} as with taking $(B_{-1},C_{-1})$ and $(B'_{-1},C'_{-1})$, respectively. The first statement now follows from Theorem $\ref{thm:7}$. The last statement follows from a similar argument and Theorem \ref{thm:8}.
\end{proof}
 
\begin{example}
There are four classes of upper $9$-gap balancing pairs whose initial terms are $(9,19)$, $(14,31)$, $(20,47)$, and $(33,83)$, respectively. Note that $(9,19)$ and $(33,83)$ appear in conjugate classes. Similarly, $(14,31)$ and $(20,47)$ appear in conjugate classes. Four transition functions can be defined between the four classes to sort the upper $9$-gap balancing number in ascending order as shown in Table $\ref{tab:2}$. Asterisks denote non-integer values. In particular, 
\begin{equation*}
\begin{aligned}
t_1(x) &= \frac{27x+5y-16}{23}, \\ 
t_2(x) &= \frac{177x+26y-64}{161}, \\ 
t_3(x) &= \frac{27x+5y-16}{23}, \\ 
t_4(x) &= \frac{163x+9y-8}{161}, 
\end{aligned}
\
\begin{aligned}
\hat{t}_1(x) &= \frac{40x+27y-160}{23}, \\ 
\hat{t}_2(x) &= \frac{208x+177y-832}{161}, \\ 
\hat{t}_3(x) &= \frac{40x+27y-160}{23}, \\ 
\hat{t}_4(x) &= \frac{72x+163y-288}{161},
\end{aligned}
\ 
\begin{aligned}
T_1(x) &= \frac{27x+5y+18}{23}, \\ 
T_2(x) &= \frac{177x+26y+72}{161}, \\ 
T_3(x) &= \frac{27x+5y+18}{23}, \\ 
T_4(x) &= \frac{163x+9y+9}{161}. 
\end{aligned}
\end{equation*}
Here $t_1=t_3$, $\hat{t}_1=\hat{t}_3$, $T_1=T_3$, and $\widehat{T}_1=\widehat{T}_3$ (not shown above) as a consequence of Theorem \ref{thm:12}. None of the sequences or subsequences in Table \ref{tab:2} appear in The On-line Encyclopedia of Integer Sequences \cite{OEIS}.
\end{example}

\begin{table}\caption{Initial upper $9$-gap balancing numbers and related sequences.}\label{tab:2}
$$\begin{array}{r|rrrr|rrrr|rrrr}
i & 0_a & 0_b & 0_c & 0_d & 1_a & 1_b & 1_c & 1_d & 2_a & 2_b & 2_c & 2_d   \\ \hline
B & 9 & 14 & 20 & 33 & 38 & 65 & 99 & 174 & 203 & 360 & 558 & 995    \\ 
C & 19  & 31 & 47 & 83 & 97 & 173 & 269 & 481& 563 & 1007 & 1567 & 2803  \\  \hline
m & 9 & 15 & 23 & 41 & 48 & 86 & 134 & 240 & 281 & 503 & 783 & 1401  \\
r & 0 & 1 & 3 & 8  & 10 & 21 & 35 & 66  & 78  & 143 & 225 & 406   \\
\hat{r} & 1 & 9 & 17 & 33 & 39 & 71 & 111 & 199 & 233 & 417 & 649 & 1161  \\ \hline
t_1 & 14 & * & 33 & * & 65 & * & 174 & * & 360 & *& 995 & *  \\ 
t_2 & * & 20 & * & * & * & 99  & * & * & * & 558 & *  & *   \\ 
t_4 & * & * & * & 38 & * & * & * & 203  & * & * & * & 1164 \\ \hline
f_9 & 38 & 65 & 99 & 174 & 203 & 360 & 558 & 995  & 1164  & 2079 & 3233 & 5780 
 \end{array}$$
 \end{table}

\section{Recursive formulas and other results}

In this section we give recursive formulas, generating functions, and other results involving upper $k$-gap balancing numbers appearing in a particular class. This generalizes and unifies previous work on gap balancing numbers. New limits are presented in Theorem \ref{thm:15} and Corollary \ref{cor:1} that involve both upper $k$-gap balancing numbers and their corresponding upper $k$-gap Lucas-balancing numbers. Also Cassini-like formulas are given for upper $k$-gap balancing numbers and related sequences.

Dash et al.~\cite{tbal} established recursive formulas analogous to the first three for lower $k$-gap balancing numbers.

\begin{theorem}\label{thm:13}
Let $((B_i,C_i))_{i\geq 0}$ be a class of upper $k$-gap balancing pairs with associated upper $k$-gap balancer pairs $((r_i,\hat{r}_i))_{i\geq0}$, and counterbalancers $(m_i)_{i\geq0}$. Then
\begin{enumerate}
\item[(a)] $B_{i+1}= 6B_i-B_{i-1}+2-2k$;
\item[(b)] $C_{i+1} = 6C_i-C_{i-1}$;
\item[(c)] $r_{i+1}=6r_i-r_{i-1}+2k$;
\item[(d)] $\hat{r}_{i+1}=6\hat{r}_i-\hat{r}_{i-1}$;
\item[(e)] $m_{i+1}=6m_i-m_{i-1}+2$.
\end{enumerate}
\end{theorem}

\begin{proof}
From \eqref{eq:10} and \eqref{eq:12}, it follows that ${B_{i+1} = 3B_i+C_i+1-k}$ and ${B_{i-1} = 3B_i-C_i+1-k}$. Thus $B_{i+1}+B_{i-1}=6B_i+2-2k$ which is equivalent to (a). The proofs for (b), (c), and (d) are similar. Lastly, use Definition \ref{def:3} with (a) and (c) to see
\begin{equation*}
m_{i+1}=B_{i+1}+r_{i+1}=6B_i-B_{i-1}+2-2k+6r_i-r_{i-1}+2k=6m_i-m_{i-1}+2.
\end{equation*}
\end{proof}

The recurrence relations in Theorem \ref{thm:13} can be used to determine the limits of the ratio of successive terms in a class of upper $k$-gap balancing numbers and associated sequences. The following theorem generalizes the corresponding result for balancing numbers \cite[Thm.~8.1, p.~103]{bal}. Its proof is similar and omitted.

\begin{theorem}
Let $((B_i,C_i))_{i\geq0}$ be a class of upper $k$-gap balancing pairs, $((r_i,\hat{r}_i))_{i \geq 0}$ its upper $k$-gap balancer pairs, and $(m_i)_{i\geq 0}$ the associated counterbalancers. Then
\begin{equation*}
\lim_{i\rightarrow\infty} \frac{B_{i+1}}{B_i}=\lim_{i\rightarrow\infty} \frac{C_{i+1}}{C_i}=\lim_{i\rightarrow\infty} \frac{r_{i+1}}{r_i}=\lim_{i\rightarrow\infty} \frac{\hat{r}_{i+1}}{\hat{r}_i}=\lim_{i\rightarrow\infty} \frac{m_{i+1}}{m_i}=3+\sqrt{8}.
\end{equation*}
\end{theorem}

The next limits involve terms from two sequences. 
\begin{theorem}\label{thm:15}
Let $((B_i,C_i))_{i\geq0}$ be a class of upper $k$-gap balancing pairs and $((r_i,\hat{r}_i))_{i \geq 0}$ its upper $k$-gap balancer pairs. Then
\begin{equation*}
\lim_{i\rightarrow\infty} \left(C_i-\sqrt{8}B_i\right)=\sqrt{2}(1-k) \ \textrm{ and } \ \lim_{i\rightarrow\infty} \left(\hat{r}_i-\sqrt{8}r_i\right)=\sqrt{2}k.
\end{equation*}
\end{theorem}

\begin{proof}
Completing the square in \eqref{eq:5} and rearranging, we see
\begin{equation*}
C_i=\sqrt{8}\left(B_i-\frac{k-1}{2}\right)\sqrt{1+\frac{2k^2-1}{8\left(B_i-\frac{k-1}{2}\right)^2}}
\end{equation*}
since $B_i-\frac{k-1}{2}>0$ (cf.~Example \ref{ex:2}). Applying the inequalities $1\leq \sqrt{1+x}\leq 1+\frac{x}{2}$ when $k\geq 1$ yields 
\begin{equation*}
\sqrt{8}\left(B_i-\frac{k-1}{2}\right)\leq C_i \leq \sqrt{8}\left(B_i-\frac{k-1}{2}\right)\left[1+\frac{2k^2-1}{16\left(B_i-\frac{k-1}{2}\right)^2}\right]
\end{equation*}
for sufficiently large $B_i$ since $B_i\rightarrow \infty$ as $i\rightarrow \infty$. It follows that \begin{equation*}
0\leq C_i-\sqrt{8}\left(B_i-\frac{k-1}{2}\right)\leq \frac{\sqrt{8}(2k^2-1)}{16\left(B_i-\frac{k-1}{2}\right)}
\end{equation*}
so that 
\begin{equation*}
\lim_{i\rightarrow \infty} \left(C_i-\sqrt{8}\left(B_i-\frac{k-1}{2}\right)\right)=0 
\end{equation*}
which is equivalent to the stated limit. The $k=0$ case follows from by adapting the argument to use the inequalities $1+x\leq \sqrt{1+x}\leq 1+\frac{x}{2}$. The second limit can be deduced similarly using \eqref{eq:6}.
\end{proof}

Dividing in the limits from Theorem $\ref{thm:15}$ by $B_i$ and $r_i$, respectively, gives the following result.

\begin{corollary}\label{cor:1}
Let $((B_i,C_i))_{i\geq0}$ be a class of upper $k$-gap balancing pairs and $((r_i,\hat{r}_i))_{i \geq 0}$ its upper $k$-gap balancer pairs. Then
\begin{equation*}
\lim_{i\rightarrow\infty}\frac{C_i}{B_i}=\lim_{i\rightarrow\infty}\frac{\hat{r}_i}{r_i}=\sqrt{8}.
\end{equation*}
\end{corollary} 

Next we present the generating function for a class of upper $k$-gap balancing numbers and give an example. Its proof is similar to the corresponding result for cobalancing numbers  \cite[Thm.~4.1, pp.~1193--1194]{cobal} and is omitted.

\begin{theorem}\label{thm:16}
Let $((B_i,C_i))_{i\geq0}$ be a class of upper $k$-gap balancing pairs. The generating function for its upper $k$-gap balancing numbers is 
\begin{equation*}
 G(s)  = \frac{(2-2k-B_1+6B_0)s^2+(B_1-7B_0)s+B_0}{(1-s)(1-6s+s^2)}.
 \end{equation*}
\end{theorem}

Given $k\geq 0$, let $n$ be the number of classes of upper $k$-gap balancing numbers. Theorem \ref{thm:16} can be used to determine the generating function $G_i(s)$ for each of the $n$ classes.  If the classes are labeled so that their initial upper $k$-gap balancing numbers occur in ascending order, the generating function for all upper $k$-gap balancing numbers is
\begin{equation*}
G(s)=\sum_{i=1}^n s^{i-1}G_i(s^n).
\end{equation*}

\begin{example}
There are four classes of upper $9$-gap balancing numbers whose generating functions are
\begin{equation*}
\begin{aligned}
G_1(s) &=\frac{-25s+9}{(1-s)(1-6s+s^2)}, \quad &  G_2(s) &=\frac{3s^2-33s+14}{(1-s)(1-6s+s^2)}, \\
G_3(s) &=\frac{5s^2-41s+20}{(1-s)(1-6s+s^2)}, & G_4(s) &= \frac{8s^2-57s+33}{(1-s)(1-6s+s^2)}, \end{aligned}
\end{equation*}
using Theorem \ref{thm:16}. The generating function for all upper $9$-gap balancing numbers is
\begin{equation*}
G(s)=\frac{-8s^8+3s^7+2s^6+3s^5+49s^4-13s^3-6s^2-5s-9}{(s-1)(s^8-6s^4+1)}.
\end{equation*}
\end{example}

The first formula of the last theorem is a generalization of an identity of Panda and Ray for cobalancing numbers \cite[Thm.~3.2(a), p.~1192]{cobal}. The other Cassini-like formulas (cf.~\cite[Thm.~5.3,  pp.~74--75]{Koshy}) are new. We give a different proof than Panda and Ray which avoids induction.

\begin{theorem}\label{thm:6.7}
Let $((B_i,C_i))_{i\geq0}$ be a class of upper $k$-gap balancing pairs, $((r_i,\hat{r}_i))_{i \geq 0}$ its upper $k$-gap balancer pairs, and $(m_i)_{i\geq 0}$ the associated counterbalancers. Then
\begin{enumerate}
\item[(a)] $(B_i+k-1)^2-B_{i-1}B_{i+1}=(2k-1)^2$;
\item[(b)] $C_i^2-C_{i-1}C_{i+1}=-8(2k^2-1)$;
\item[(c)] $(r_i-k)^2-r_{i-1}r_{i+1}=1$;
\item[(d)] $\hat{r}_i^2-\hat{r}_{i-1}\hat{r}_{i+1}=8(2k^2-1)$;
\item[(e)] $(m_i-1)^2-m_{i-1}m_{i+1}=-4(k^2-1)$.
\end{enumerate}
\end{theorem}

\begin{proof}
Using \eqref{eq:10} and \eqref{eq:12}, we obtain (a) from observing that
\begin{align*}
(B_i+k-1)^2-B_{i-1}B_{i+1} &= (B_i+k-1)^2-(3B_i-C_i+1-k)(3B_i+C_i+1-k) \\
&= C_i^2-[8B_i^2+8(1-k)B_i+(2k-1)^2]+(2k-1)^2 \\
&= (2k-1)^2 
\end{align*}
where the last equality follows from \eqref{eq:5}. The other formulas can be established similarly where the identity $m=\frac{C-1}{2}$ is used for (e).  
\end{proof}
Considering the expressions on the right side of the formulas in Theorem \ref{thm:6.7} as sequences indexed by $k$, we note several connections with The On-line Encyclopedia of Integer Sequences \cite{OEIS}. The sequences \seqnum{A016754} and \seqnum{A000012} appear in (a) and (c), respectively. A constant multiple of the sequence \seqnum{A056220} arises for (b) and (d), and a constant multiple of \seqnum{A005563} occurs in (e).

\section*{Acknowledgements}
The authors thank Professor Jerry Metzger for his help in exploring several avenues of approach, providing computational expertise, and for many fruitful conversations during work on this paper.

\end{document}